\documentclass[11pt]{article}

\usepackage{latexsym}
\usepackage{mathrsfs}
\usepackage{amsmath,amsthm}
\usepackage{graphicx,epstopdf,epsfig,multirow,epic,bm}
\usepackage{amssymb}
\usepackage{color}
\usepackage{colortbl}

\theoremstyle{plain}
\newtheorem{thm}{Theorem}
\newtheorem{cor}[thm]{Corollary}
\newtheorem{lem}[thm]{Lemma}

\theoremstyle{definition}

\theoremstyle{definition}

\theoremstyle{plain}
\newtheorem{conj}{Conjecture}
\theoremstyle{plain}

\normalsize \rm
\DeclareGraphicsExtensions{.eps,.eps.gz}
\theoremstyle{plain}

\newcommand{\qqed}{\hfill $\square$}
\newcommand{\ud}{\mathrm{deg}}

\usepackage[displaymath,mathlines]{lineno}

\begin{document}

\thispagestyle{empty}

\begin{center}
{\Large\bf A Note On Vertex Distinguishing Edge colorings of Trees}\\[8pt]

{\large Songling \textsc{Shan}$^a$,\quad Bing \textsc{Yao}$^{b,}$\footnote{Corresponding author, Email: yybb918@163.com}, }\\[8pt]

{\footnotesize \emph{a.} Department of Mathematics and Statistics of Georgia State
University, Atlanta, 30303-3083, USA.\\
E-mail: ld5772156649@163.com\\
\emph{b.} College of Mathematics and Statistics, Northwest
Normal University, Lanzhou, 730070, P.R.China
}
\end{center}

\begin{abstract}
A proper edge coloring of a simple graph $G$ is called a vertex
distinguishing edge coloring (vdec) if for any two distinct vertices $u$
and $v$ of $G$, the set of the colors assigned to the edges incident
to $u$ differs from the set of the colors assigned to the edges
incident to $v$. The minimum number of colors required for
all vdecs of $G$ is denoted by $\chi\,'_s(G)$ called the vdec chromatic number of $G$. Let $n_d(G)$ denote the number of vertices of degree $d$ in $G$. In this note, we show that a tree $T$ with $n_2(T)\leq n_1(T)$ holds $\chi\,'_s(T)=n_1(T)+1$ if its diameter $D(T)=3$ or one of two particular
trees with $D(T) =4$, and $\chi\,'_s(T)=n_1(T)$ otherwise; furthermore $\chi\,'_{es}(T)=\chi\,'_s(T)$ when $|E(T)|\leq 2(n_1(T)+1)$, where $\chi\,'_{es}(T)$ is the equitable vdec chromatic number of $T$.\\[6pt]
\textbf{AMS Subject Classification (2000):} 05C15\\[6pt] \textbf{Keywords:}
vertex distinguishing coloring, edge coloring, trees
\end{abstract}

\section{Introduction and concepts}

Labeled graphs are becoming an increasingly useful family of
mathematical models for a broad range of applications, such as
\emph{time tabling and scheduling}, \emph{frequency assignment},
\emph{register allocation}, \emph{computer security} and so on. In
\cite{Burris-Schelp1997}, Burris and Schelp introduced that a proper
edge $k$-coloring of a simple graph $G$ is called a \emph{vertex
distinguishing edge $k$-coloring} ($k$-vdec, or vdec for short) if for any two distinct
vertices $u$ and $v$ of $G$, the set of the colors assigned to the
edges incident to $u$ differs from the set of the colors assigned to
the edges incident to $v$. The minimum number of colors required for
a vertex distinguishing coloring of $G$ is denoted by $\chi
\,'_s(G)$. The maximum and minimum degrees of $G$ are denoted by $\Delta (G)$ and $\delta(G)$, respectively. Let $n_d(G)$ denote the number of vertices of degree $d$ with respect to $\delta(G)\leq d\leq
\Delta(G)$ in $G$, or write $n_d=n_d(G)$ if there is no confusion.
Burris and Schelp \cite{Burris-Schelp1997} presented the following
conjecture:

\begin{conj} \label{conj:c4-Burris-Schelp-conjecture}
Let $G$ be a simple graph having no isolated edges and at most one isolated vertex, and let $k$ be the smallest integer such that $(^k_d)\geq n_d$ for all $d$ with respect to $\delta(G)\leq d\leq
\Delta(G)$. Then $k\leq \chi \,'_s(G)\leq k+1$.
\end{conj}

It seems very difficult to settle down Conjecture \ref{conj:c4-Burris-Schelp-conjecture}, since no more results to verify it (cf. \cite{Balister-Bollobas-Schelp}, \cite{Burris-Schelp1997}, \cite{Cerny-Hornak-Sotak1996}). We show some results on trees to confirm positively Conjecture \ref{conj:c4-Burris-Schelp-conjecture}, and try to approximate $\chi\,'_s(G)$ of a graph $G$ by $\chi\,'_s(T)$ of some tree $T$ generated from $G$. Graphs mentioned here are finite, simple and undirected. We use standard terminology and notation of graph theory, and write $[m,n]=\{m,m+1,\dots ,n\}$ for integers $m,n$ with respect to $0\leq m<n$. Let $\pi$ be a $k$-$vdec$ of a graph $G$, and $S_i=\{uv:\pi(uv)=i,uv\in
E(G)\}$ for $i\in [1, k]$. We call $\pi$ an \emph{equitable}
$k$-\emph{vdec} if $S_i$ and $S_j$ differ in size at most one for
distinct $i,j\in [1, k]$. The smallest value of $k$ such that $G$ has
an equitable $k$-\emph{vdec} is denoted by $\chi\,'_{es}(G)$. Clearly, $\chi\,'_{s}(G)\leq \chi\,'_{es}(G)$. The set of neighbors of a vertex $u$ of $G$ is denoted as $N_G(u)$, or $N(u)$ if no confusion. The diameter of $G$ is written as $D(G)$. A \emph{leaf} is a vertex of degree one, and a \emph{$k$-degree vertex} is one of degree $k\geq 2$. $P_n$ is a path of length $n-1$.

A tree $Q$ of diameter four has its own vertex set $V(Q)=W_r\cup X_m\cup \big (\bigcup ^n_{i=1}Y_i\big )$, where $W_r=\{w_0,w_i:i\in [1,r]\}$ if $r\geq 1$, $X_m=\{s_i,s\, '_i:i\in [1,m]\}$ if $m\geq 1$, $Y_i=\{t_i,t\, '_{i,j}:j\in [1,r_i]\}$ for $r_i\geq 2$ and $i\in [1,n]$ if $n\geq 1$; and $Q$ has its own edge set $E(Q)=E(W_r)\cup E(X_m)\cup \big (\bigcup ^n_{i=1}E(Y_i)\big )$, where $E(W_r)=\{w_0w_i:i\in [1,r]\}$, $E(X_m)=\{w_0s_i,s_is\, '_i:i\in [1,m]\}$, $E(Y_i)=\{w_0t_i,t_it\, '_{i,j}:j\in [1,r_i]\}$ for $r_i\geq 2$ and $i\in [1,n]$. Thereby, we can write this tree $Q$ having diameter $D(Q)=4$ as $Q=Q(r,m,n)$ with the center $w_0$ hereafter. Especially, $Q(0,2,0)=P_5$ such that  $\chi\,'_s(Q(0,2,0))=n_1(Q(0,2,0))+2$.  let $U=\{Q(0,2,0)$, $Q(r,2,0)$, $Q(0,1,1)\}$. We will show the following results in this note.

\begin{thm} \label{them:main-theorem}
Let $T$ be a tree with $n_2(T)\leq n_1(T)$ and at least three vertices. Then $\chi\,'_s(T)=n_1(T)+1$ if $D(T)=3$ and $T\in U\setminus \{Q(0,2,0)\}$; $\chi\,'_s(T)=n_1(T)+2$ if $T=Q(0,2,0)$; and $\chi\,'_s(T)=n_1(T)$ if $D(T)\geq 4$ and $T\not \in U$.
\end{thm}

\begin{thm} \label{thm:GLT2-theorem-222}
Let $T$ be a tree with $n_2(T)\leq n_1(T)$ and at least three vertices. If $|E(T)|\leq 2[n_1(T)+1]$, then $\chi\,'_{es}(T)=\chi\,'_{s}(T)$.
\end{thm}

\section{Lemmas and proofs}

\begin{lem} \label{them:lemma-11}
Let $T$ be a tree with $n_2(T)\geq n_1(T)$. Then there exists a $2$-degree vertex such that one of its neighbors is either a leaf or a vertex of degree $2$.
\end{lem}
\begin{proof} Write $n_d=n_d(T)$ for $d\in [\delta(T),\Delta(T)]$, here $n_d=n_d(T)=0$ if $T$ has no vertices of degree $d$. Let $n_{\geq k}=|V(T)|-n_1-n_2-\cdots -n_{k-1}$ for $k\geq 2$. By contradiction. Assume that each $2$-degree vertex $x$ has its neighborhood $N(x)=\{x_1,x_2\}$ such that degree $\ud_T(x_i)\geq 3$ for $i=1,2$. So, we have $n_2\leq n_{\geq 3}$ by the assumption. Applying the formula $n_1=2+\sum _{3\leq d\leq \Delta(T)}(d-2)n_d$ shown in the article \cite{Yao-Zhang-Wang2010}, we have $n_1\geq 2+n_d$ for $d\geq 3$, and $n_1\geq 2+n_3+2n_{\geq 4}=2+2n_{\geq 3}-n_3$, which shows that $n_{\geq 3}\leq \frac{1}{2}(n_1+n_3)-1$. Thereby, the following inequalities
$$n_2\leq n_{\geq 3}\leq \frac{1}{2}(n_1+n_3)-1\leq n_1-2$$
show a contradiction.
\end{proof}

\begin{lem} \label{them:lemma-aa}
Let $T$ be a tree with diameter three. Then $\chi\,'_{s}(T)=n_1(T)+1$, and $\chi\,'_{es}(T)=\chi\,'_{s}(T)$.
\end{lem}
\begin{proof} Let two complete graphs $K_{1,m}$ and $K_{1,n}$ have their own vertex and edge sets as follows: $V(K_{1,m})=\{s,s_1,s_2,\dots,s_m\}$ and $E(K_{1,m})=\{ss_1,ss_2,\dots,ss_m\}$, $V(K_{1,n})=\{t,t_1,t_2,\dots,t_n\}$ and $E(K_{1,n})=\{tt_1,tt_2,\dots,tt_n\}$, where $s,t$ are the centers of $K_{1,m}$ and $K_{1,n}$. Each tree of diameter 3, denoted as $S_{m+1,n+1}$, can be obtained by joining two centers of $K_{1,m}$ and $K_{1,n}$ with an edge. So,  $S_{m+1,n+1}$ has its own vertex set $V(S_{m+1,n+1})=V(K_{1,m})\cup V(K_{1,n})$ and edge set $E(S_{m+1,n+1})=E(K_{1,m})\cup E(K_{1,n})\cup \{st\}$. Suppose that $f$ is a $k$-vdec of $S_{m+1,n+1}$ such that $k=\chi\,'_{s}(S_{m+1,n+1})$. Let $C(u,f)$ be the set of the colors assigned to the edges incident to $u$ of $S_{m+1,n+1}$. Since $C(s_i,f)\neq C(s_j,f)$, $C(s_i,f)\neq C(t_l,f)$, $C(t_i,f)\neq C(t_j,f)$, and $C(s,f)\neq C(t,f)$, so we know $k\geq m+n+1=n_1(S_{m+1,n+1})+1$. This $k$-vdec $f$ can be exactly defined as: $f(ss_i)=i$ for $i\in [1,m]$; $f(st)=m+1$; $f(tt_j)=m+1+j$ for $j\in [1,n]$. Thereby, $k=m+n+1$. By the definition of the $k$-vdec $f$, we are not hard to see $\chi\,'_{es}(T)=\chi\,'_{s}(T)$.
\end{proof}

\begin{lem} \label{them:lemma-bb}
Let $U=\{Q(0,2,0)$, $Q(r,2,0)$, $Q(0,1,1)\}$. For all trees $Q=Q(r,m,n)$ of diameter four we have $\chi\,'_s(Q)=n_1(Q)+1$ if $Q\in U\setminus \{Q(0,2,0)\}$; $\chi\,'_s(Q)=n_1(Q)+2$ if $Q=Q(0,2,0)$; and $\chi\,'_s(Q)=n_1(Q)$ for $Q\not \in U$. Furthermore,  $\chi\,'_{es}(Q)=\chi\,'_{s}(Q)$.
\end{lem}
\begin{proof} Using the description of a tree $Q=Q(r,m,n)$ of diameter four in Section 1. Let $n_1=n_1(Q)$, and $f$ be a $k$-vdec of $Q$ such that $k=\chi\,'_s(Q)$. Note that $D(Q)=4$.

\emph{Case A.} $(r,m,n)=(r,m,0)$ with $m\geq 2$ and $r\geq 0$. Clearly, $n_1=r+m$.

\emph{Case A1.} $(r,m,0)=(r,2,0)$, $n_1=r+2$. We can easily see $\chi\,'_s(Q)=n_1+1$ for $r\geq 0$, since we have to color one of two edges $w_0s_1$ and $w_0s_2$ with one color that is not in $\{f(w_0w_s),f(s_is\, '_i):s\in [1,r], i\in [1,2]\}$. As $r=0$, $Q$ is a path with $4$ vertices, so $\chi\,'_s(Q)=n_1+2=4$.

\emph{Case A2.} $(r,m,0)$ with $m\geq 3$, $n_1=r+m$. We show $f$ in the following: $f(s_is\, '_i)=i$ for $i\in [1,m]$; $f(w_0s_i)=i+1$ for $i\in [1,m-1]$, and $f(w_0s_m)=1$; $f(w_0w_j)=m+j$ for $j\in [1,r]$ if $r\neq 0$. Clearly, $C(u,f)\neq C(v,f)$ for any two vertices $u,v\in V(Q)$, which means  $\chi\,'_s(Q)=n_1$.

\emph{Case B.} $(r,0,n)$ with $n\geq 2$, $n_1=r+\sum^n_{i=1}r_i$. It is easy to show $\chi\,'_s(Q)=n_1$ in this case, since $n_2\leq 1$.

\emph{Case C.} $(0,m,n)$ with $m\geq 1$ and $n\geq 1$, $n_1=m+\sum^n_{i=1}r_i$.

\emph{Case C1.} $(0,m,n)=(0,1,1)$, $n_1=1+r_1$. Since $f(s_1s\, '_1)=f(w_0t_1)$ if $\chi\,'_s(Q)=n_1$ in this case, we can see $C(w_0,f)=C(s_1,f)$; a contradiction. So $\chi\,'_s(Q(0,1,1))=2+r_1=n_1+1$.

\emph{Case C1.} $(0,m,n)=(0,1,n)$ with $n\geq 2$, $n_1=1+\sum^n_{i=1}r_i$. To show $\chi\,'_s(Q)=n_1$, we have $f$ defined as: $f(s_1s\, '_1)=1$; $f(t_1t\, '_{1,j})=1+j$ for $j\in [1,r_1]$; $f(t_it\, '_{i,j})=f(t_{i-1}t\, '_{i-1,j})+j$ for $j\in [1,r_i]$ and $i\in [2,n]$; $f(w_0s_1)=f(t_{n}t\, '_{n,1})$; $f(w_0t_j)=f(t_{j+1}t\, '_{j+1,1})$ for $j\in [1,n-1]$; $f(w_0t_n)=f(s_1s\, '_1)$.

\emph{Case C2.} $(0,m,n)=(0,m,1)$ with $m\geq 2$, $n_1=m+r_1$. For showing $\chi\,'_s(Q)=n_1$ we define: $f(s_is\, '_i)=i$ for $i\in [1,m]$; $f(w_0s_i)=i+1$ for $i\in [1,m-1]$, and $f(w_0s_m)=f(t_1t\, '_{1,1})$; $f(t_1t\, '_{1,j})=m+j$ for $j\in [1,r_1]$ and $f(w_0t_1)=m$.

\emph{Case C3.} $(0,m,n)=(0,m,n)$ with $m\geq 2$ and $n\geq 1$, $n_1=m+\sum^n_{i=1}r_i$. We have $\chi\,'_s(Q)=n_1$ by defining $f$ appropriately by the methods showing in Case C1 and Case C2.

\emph{Case D.} $(r,m,n)$ with $r\geq 1$, $m\geq 1$ and $n\geq 1$, $n_1=r+m+\sum^n_{i=1}r_i$. By the techniques use in the above cases, we can define $f$ to show $\chi\,'_s(Q)=n_1$.

Through the above Cases A, B, C and D, we conclude $\chi\,'_{es}(Q)=\chi\,'_{s}(Q)$, since one color is used at most twice under the vdec $f$. The lemma is covered.
\end{proof}

\textbf{The proof of Theorem \ref{them:main-theorem}.} Let $n_i=n_i(T)$ for $i=1,2$, and neighborhoods $N(w)=N_T(w)$ for $w\in V(T)$. Because $T=K_{1,n-1}$ when $D(T)=2$, so $\chi\,'_{s}(T)=n_1=n-1$. For $D(T)=3,4$, by Lemmas \ref{them:lemma-aa} and \ref{them:lemma-bb}, the theorem holds true. So, we show $\chi\,'_{s}(T)=n_1$ by induction on vertex numbers of trees $T$ with diameter $D(T)\geq 5$. We will use the description of trees having diameter four in the following discussion.

\emph{Case 1.} There exists a leaf $v$ having a neighbor $u$ with degree $\ud_T(u)\geq 4$. Let $T\,'=T-v$, so $n_1(T\,')=n_1-1$. Clearly, $D(T\,')\geq 5$.

\emph{Case 1.1.} $n_2(T\,')\leq n_1(T\,')$. Then by induction hypothesis, there is an edge coloring $\xi_{T\,'}:E(T\,')\rightarrow \{1,2,\dots,b\,'\}$ such that $\chi\,'_s(T\,')=b\,'=n_1(T\,')=n_1-1$. It is straightforward to define a vdec $\xi_{T}$: $E(T)\rightarrow \{1,2,\dots,b\,', b\,'+1\}$ such that $\xi_{T}(e)=\xi_{T\,'}(e)$ for $e\in E(T)\setminus \{uv\}$, and $\xi_{T}(uv)=b\,'+1$.

\emph{Case 1.2.} $n_1(T\,')=n_2(T\,')-1$. By Lemma \ref{them:lemma-11}, there exists a $2$-degree vertex $x\in V(T\,')$ with $N_{T\,'}(x)=\{x_1,x_2\}$ such that $\ud_{T}(x_1)=\ud_{T\,'}(x_1)\leq 2$. Notice that $\ud_{T\,'}(x)=\ud_{T}(x)$ and $\ud_{T\,'}(x_1)=\ud_{T}(x_1)$. Let $T_1$ be the tree obtained from $T\,'$ by suppressing the vertex $x$, that is, $T_1=T\,'-x+x_1x_2$. Clearly, $D(T_1)\geq 4$, $n_1(T_1)=n_1(T\,')$ and $n_2(T_1)=n_2(T\,')-1$.

If $D(T_1)=4$, we have $T_1=Q(0,m,0)$ with $m\geq 3$, since $n_1(T_1)=n_2(T_1)$ and  $\ud_T(u)\geq 4$. Thus, $u$ is the center of $T_1=Q(0,m,0)$, the leaf $v$ is the leaf $w_1$ of $Q(1,m,0)$, and $T$ is the tree obtained by subdivision of the edge $s_1s\, '_1$. It is not difficult to see $\chi\,'_{s}(Q(1,m,0))=n_1(Q(1,m,0))$, so $\chi\,'_{s}(T)=n_1$.

If $D(T_1)\geq 5$, again by induction hypothesis, there is a vdec $\xi_{T_1}:~E(T_1)\rightarrow \{1,2,\dots,b_1\}$ such that $\chi\,'_s(T_1)=b_1=n_1(T_1)=n_1-1$. Hence, $T$ has a proper edge coloring $\xi_{T}$ defined as: $\xi_{T}(e)=\xi_{T_1}(e)$ for $e\in E(T)\setminus \{uv, xx_1,xx_2\}$, and $\xi_{T}(uv)=b_1+1$, $\xi_{T}(xx_1)=b_1+1$ and $\xi_{T}(xx_2)=\xi_{T_1}(x_1x_2)$. It is easy to check $\xi_{T}$ to be a desired $n_1$-vdec of $T$, since $C(\xi_{T},u)\neq C(\xi_{T}, x_1)$.

\emph{Case 2.} There is a leaf $v$ having a 3-degree neighbor $u$, and Case 1 is false.

\emph{Case 2.1.} $v\,'$ is another leaf in the neighborhood $N(u)=\{v,v\,',u\,'\}$, and $T$ has a $2$-degree vertex $x$ having its neighborhood $\{x_1,x_2\}$ such that $x_1$ is a leaf of $T$.

We have a tree $T_1=T-\{v,v\,',x\}+x_1x_2$. Clearly, $D(T_1)\geq 3$ since $D(T)\geq 5$, $n_1(T_1)=n_1-1$ and $n_2(T_1)=n_2-1$.

If $D(T_1)=3$, $T_1$ is a $3$-diameter tree $S_{m+1,n+1}$ (we use the description shown in the proof of Lemma \ref{them:lemma-aa}) obtained by joining two centers of $K_{1,m}$ and $K_{1,n}$ with an edge. Since $D(T)\geq 5$, so $T$ can be obtained by subdividing the edge $ss_1$ ($=x_1x_2$) to form a path $sxs_1=x_1xx_2$ of $T$ and joining $t_1$ ($=u$) to two new vertices $v$ and $v\, '$ to $S_{m+1,n+1}$. It is not hard to make a desired $n_1$-vdec of $T$ based on the structure of $T$.

If $D(T_1)=4$, then $T_1=Q(r,m,n)$. If $T_1=Q(0,2,0)$, thus it goes to $T$ such that $n_1=3$ and $n_2=4$; a contradiction. As $T_1=Q(r,m,n)\neq Q(0,2,0)$, we can show $\chi\,'_{s}(T)=n_1$ by the edge colorings used in the proof of Lemma \ref{them:lemma-bb}.

For $D(T_1)\geq 5$, $T_1$, by induction hypothesis, has a vdec $\xi_{T_1}:~E(T_1)\rightarrow \{1,2,\dots,b_1\}$ such that $\chi\,'_s(T_1)=b_1=n_1(T_1)=n_1-1$. Notice that $u$ is a leaf of $T_1$. So, we can extend $\xi_{T_1}$ to a vdec $\xi_{T}$ of $T$ as follows: $\xi_{T}(e)=\xi_{T_1}(e)$ for $e\in E(T)\setminus \{uv, uv\,',xx_1,xx_2\}$; if $\xi_{T_1}(x_1x_2)\not \in C(\xi_{T_1},u\,')$, we set $\xi_{T}(uu\,')=b_1+1$, $\xi_{T}(uv\,')=\xi_{T_1}(uu\,')$, $\xi_{T}(uv)=\xi_{T_1}(x_1x_2)$, $\xi_{T}(xx_2)=\xi_{T_1}(x_1x_2)$ and $\xi_{T}(xx_1)=b_1+1$; if  $\xi_{T_1}(x_1x_2)\in C(\xi_{T_1},u\,')$, we define $\xi_{T}(uu\,')=\xi_{T_1}(uu\,')$, $\xi_{T}(uv)=\xi_{T_1}(x_1x_2)$, $\xi_{T}(uv\,')=b_1+1$, $\xi_{T}(xx_2)=b_1+1$  and $\xi_{T}(xx_1)=\xi_{T_1}(uu\,')$.

\emph{Case 2.2.} $\ud_T(v\,')=1$ for $v\,'\in N(u)=\{v,v\,',u\,'\}$, and $T$ has no a $2$-degree vertex that is adjacent to a leaf of $T$. Let $P=p_1p_2\cdots p_{D-1}p_D$ be a longest path of $T$, where $D$ is the diameter of $T$. Here, $\ud_T(p_2)=3=\ud_T(p_{D-1})$.

Without loss of generality, $p_2\neq u$, so $p_2$ is adjacent to two leaves $p_1,p\,'_1\in N(p_2)=\{p_1,p\,'_1,p_3\}$ with $\ud_T(p_3)\neq 1$. Suppose that $x,y$ are two $2$-degree vertices of $T$. Let $N(x)=\{x_1,x_2\}$ and $N(y)=\{y_1,y_2\}$, then we know that $\ud_T(x_i)\geq 2$ and $\ud_T(y_i)\geq 2$ for $i=1,2$ by the hypothesis of Case 2.2. Make a tree $T_2=T-\{v,v\,',p_1,p\,'_1,x,y\}+\{x_1x_2,y_1y_2\}$ with $n_1(T_2)=n_1-2$ and $n_2(T_2)=n_2-2$.

If $D(T_2)=2$, so $T_2=K_{1,m}$ with $V(K_{1,m})=\{s,s_1,s_2,\dots,s_m\}$ and $E(K_{1,m})=\{ss_1,ss_2,\dots,ss_m\}$ for $m\geq 2$. Without loss of generality, $u=s_1$, $p_2=s_2$, $ss_1=x_1x_2$ and $ss_j=y_1y_2$ for $j\neq 1$ since $D(T)\geq 5$. So, $D(T)=5$ or $6$, $n_1\geq 4$ and $n_2=2$. The structure of $T$ enables us to show $\chi\,'_{s}(T)=n_1$ by defining a desired $n_1$-vdec of $T$.

Consider the case $D(T_2)=3$. As $T_2=P_4$, $T$ is obtained by replacing by a path $w_0xx_1yt_1$ the edge $w_0t_1$ of $Q(0,0,2)$ with $r_1=r_2=2$ ($t_2=u$, $t_1=p_2$), so $n_1=n_2=4$. It is easy to define a desired $4$-vdec of $T$ by means of the techniques used in the proof of Lemma  \ref{them:lemma-bb}. When $T_2=S_{m+1,n+1}$ with $m+n\geq 3$, the structure of $T$ is clear such that we can show easily $\chi\,'_{s}(T)=n_1$.

For $D(T_2)=4$, thus $T_2=Q(r,0,n)$ with $0\leq r\leq 1$ and $n\geq 2$ according to $T$ has no a $2$-degree vertex that is adjacent to a leaf of $T$. We have two subcases $T_2=Q(1,0,2)$ and $T_2=Q(0,0,n)$ with $n\geq 3$. So, $x_1x_2=w_0t_1$ and $y_1y_2=w_0t_2$, and $u$ and $p_2$ both are the leaves of $T_2=Q(1,0,2)$ or $T_2=Q(0,0,n)$ with $n\geq 3$. We claim $\chi\,'_{s}(T)=n_1$ by the structure of $T$ in this situation.

When $D(T_2)\geq 5$, by induction hypothesis, we take a vdec $\xi_{T_2}:~E(T_2)\rightarrow \{1,2,\dots,b_2\}$ having $\chi\,'_s(T_2)=b_2=n_1(T_2)=n_1-2$, and extend $\xi_{T_2}$ to a proper edge coloring $\xi_{T}$ of $T$ by setting $\xi_{T}(e)=\xi_{T_2}(e)$ for $e\in E(T)\setminus \{uv, uv\,',uu\,',xx_1,xx_2, p_2p_1,p_2p\,'_1,p_2p_3,yy_1,yy_2\}$, and $\xi_{T}(uu\,')=b_2+2$, $\xi_{T}(uv\,')=\xi_{T_2}(uu\,')$, $\xi_{T}(uv)=b_2+1$, $\xi_{T}(xx_2)=\xi_{T_2}(x_1x_2)$, and $\xi_{T}(xx_1)=b_2+1$; $\xi_{T}(p_2p_3)=b_2+1$, $\xi_{T}(p_2p\,'_1)=\xi_{T_2}(p_2p_3)$, $\xi_{T}(p_2p_1)=b_2+2$, $\xi_{T}(yy_2)=\xi_{T_2}(y_1y_2)$, and $\xi_{T}(yy_1)=b_2+2$. Since $C(\xi_{T_2},p_3)\setminus \{\xi_{T_2}(p_2p_3)\}$ is not equal to one of $C(\xi_{T_2},x_i)\setminus \{\xi_{T_2}(x_1x_2)\}$ for $i=1,2$, and $C(\xi_{T_2},u)\setminus \{\xi_{T_2}(uu\,')\}$ is not equal to one of $C(\xi_{T_2},y_i)\setminus \{\xi_{T_2}(y_1y_2)\}$ for $i=1,2$, so $C(\xi_{T},p_3)\neq C(\xi_{T},x_1)$ and $C(\xi_{T},u\,')\neq C(\xi_{T},y_1)$, without loss of generality. Hence, it is not hard to verify that $\xi_{T}$ is a vdec of $T$ such that $\chi\,'_s(T)= b_2+2=n_1$.

\emph{Case 2.3.} The above Case 2.1 and Case 2.2 are false, simultaneously. So, any leaf of $T$ is either adjacent to a $2$-degree vertex or a 3-degree vertex having two neighbors of non-leaves. Thereby, $\ud_T(v)=1$, $\ud_T(v\,')\geq 2$ and $\ud_T(u\,')\geq 2$ in $N(u)=\{v,v\,',u\,'\}$. Let $P=x_1x_2\cdots x_{D-1}x_D$ be a longest path of $T$ with $\ud_T(x_2)=2$ or $\ud_T(x_{D-1})=2$, where $D=D(T)$. Without loss of generality, $x_2$ has its two neighbors $x_1$ and $x_3$ holding $\ud_T(x_1)=1$ and $\ud_T(x_3)\geq 2$.

We have a tree $T_3=T-\{v,u,x_2\}+\{x_1x_3,v\,'u\,'\}$. Clearly, $D(T_3)\geq 3$ since $D(T)\geq 5$, $n_1(T_3)=n_1-1$ and $n_2(T_3)=n_2-1$. If $D(T_3)=3$, the possible case is $T_3=P_4=x_1x_3x_4x_5=x_1u\, 'v\, 'x_5$, which implies $T=x_1x_2u\, 'uv\, 'x_5+uv$. Thereby, we have $\chi\,'_{s}(T)=n_1$. If $D(T_3)=4$, then  $T_3=Q(1,2,0)$ or $T_3=Q(0,m,0)$ with $m\geq 3$ ($T_3=Q(0,2,0)$ will induce $T$ has $n_1=3$ and $n_2=4$; a contradiction). Therefore, $s_1=u\,'$, $w_0=v\,'$, and $x_1$ is a leaf of $T_3=Q(1,2,0)$ or $T_3=Q(0,m,0)$ with $m\geq 3$. The structure of $T$ is clear when $D(T_3)=4$, so we can show a $n_1$-vdec  of $T$ by the methods used in the proof of Lemma  \ref{them:lemma-bb}.

For $D(T_3)\geq 5$, by induction hypothesis, $T_3$ admits a vdec $\xi_{T_3}:~E(T_3)\rightarrow \{1,2,\dots,b_3\}$ such that $\chi\,'_s(T_3)=b_3=n_1(T_3)=n_1-1$. It is straightforward to define a $n_1$-vdec $\xi_{T}$ of $T$ as follows: $\xi_{T}(e)=\xi_{T_3}(e)$ for $e\in E(T)\setminus \{uv, uv\,',uu\,',x_2x_1,x_2x_3\}$, and $\xi_{T}(uu\,')=b_3+1$ if $\xi_{T_3}(x_1x_3)\not \in C(\xi_{T_3},u\,')$, $\xi_{T}(uv\,')=\xi_{T_3}(u\,'v\,')$, $\xi_{T}(uv)=\xi_{T_3}(x_1x_3)$, $\xi_{T}(x_2x_3)=\xi_{T_3}(x_1x_3)$, and $\xi_{T}(x_2x_1)=b_3+1$.

\emph{Case 3.} The above Case 1 and Case 2 do not appear simultaneously, and every leaf is adjacent to a $2$-degree vertex in $T$, which means $n_2=n_1$.

\emph{Case 3.1.} Each $2$-degree vertex $x$ has its neighborhood $N(x)=\{x_1,x_2\}$ such that $\ud_T(x_1)=1$ and $\ud_T(x_2)\geq 3$, and no two $2$-degree vertices have a common neighbor. Let $R=p_1p_2\cdots p_l$ be a longest path in $T$, where $l=D(T)$. Thereby, $\ud_T(p_3)\geq 3$ and $\ud_T(p_{l-2})\geq 3$ since  $n_2=n_1$. Notice that $\ud_T(x)\leq 2$ for $x\in N(p_3)\setminus \{p_2,p_4\}$ (resp. $x\in N(p_{l-2})\setminus \{p_{l-3},p_{l-1}\}$) because $x\not \in V(R)$, and furthermore this vertex $x$ is neither a leaf (Case 2 has been assumed to disappear) nor a $2$-degree vertex (by the hypothesis of this subcase and $n_2=n_1$). We claim that this subcase does not exist.

\emph{Case 3.2.} Case 3.1 does not exist at all. We have a subgraph $H$ with $V(H)=\{w,u\}\cup V\,'$ with $m\geq 2$ and $V\,'=\{y_i,x_i:i\in [1,m]\}$ and $E(H)=\{wu, uy_i, y_ix_i: i\in [1,m]\}$, where $\ud_T(w)\geq 2$, every $y_i$ is a $2$-degree vertex and every $x_i$ is a leaf for $i\in [1,m]$. Then we have a tree $T_1=T-V\,'$ such that  $D(T_1)\geq 3$ because $D(T)\geq 5$.

If $D(T_1)=3$, the possible structure is $T_1=P_4$, but, which implies $n_2=n_1+1$; a contradiction.

For $D(T_1)=4$, thus $T_1=Q(1,m,0)$ with $m\geq 2$ under the restriction of Case 3.2, and $s\, '_1=u$. The clear structure of $T$ enables us to show $\chi\,'_{s}(T)=n_1$.

For $D(T_1)\geq 5$, $n_1(T_1)=n_1-m+1> n_2(T_1)=n_2-m$ from $n_1=n_2$, and $u$ is a leaf of $T_1$. By induction hypothesis, $T_1$ admits a vdec $\xi_{T_1}:~E(T_1)\rightarrow S=\{1,2,\dots,b_1\}$ having $\chi\,'_s(T_1)=b_1= n_1(T_1)=n_1-(m-1)$. We can define a $n_1$-vdec $\xi_{T}$ of $T$ as follows: $\xi_{T}(e)=\xi_{T_1}(e)$ for $e\in E(T)\setminus E(H)$, and $\xi_{T}(wu)=\xi_{T_1}(wu)$; $\xi_{T}(uy_i)=b_1+i$ for $i\in [1,m-1]$, and $\xi_{T}(uy_m)\in S\setminus \{\xi_{T_1}(wu)\}$; $\xi_{T}(y_1x_1)=\xi_{T_1}(wu)$, $\xi_{T}(y_jx_j)=b_1+j-1$ for $j\in [2,m]$.

Thereby, Theorem \ref{them:main-theorem} follows from the principle of induction.\qqed

\vskip 0.2cm

\textbf{Proof of Theorem \ref{thm:GLT2-theorem-222}.} It is not hard to let the number of a color be one or two in the desired vdecs $\xi_{T}$ of $T$ in all cases of the proof of Theorem \ref{them:main-theorem}, since $|E(T)|\leq 2(n_1+1)$. Thereby, Theorem \ref{thm:GLT2-theorem-222} follows from the principle of induction.\qqed

\begin{cor} \label{thm:GLT2-111}
Let $G$ be a connected graph having cycles, $p$ vertices and $q$ edges. If $n_2(G)\leq 2(q-p+1)+n_1(G)$, then $\chi\,'_{s}(G)\leq 2(q-p+1)+n_1(G)$.
\end{cor}
\begin{proof} Let $H$ be a spanning tree of $G$,
so $E\,'=E(G)\setminus E(H)$. Hence, we have another tree $T$
obtained by deleting every edge $uv\in E\,'$, and then adding two
new vertices $u\,',v\,'$ by joining $u\,'$ with $u$ and $v\,'$ with
$v$ simultaneously. Clearly, $n_1(T)=2(q-p+1)+n_1(G)$, and $n_2(T)=n_2(G)$ and
$\Delta(T)=\Delta(G)$. On the other hand, each $k$-vdec $\pi$
of $G$ with $k=\chi\,'_{s}(G)$ corresponds to an edge coloring $\pi\,'$ of $T$ such that
for any two distinct
non-leaf vertices $u$ and $v$ of $T$, the set of the colors assigned to the
edges incident to $u$ differs from the set of the colors assigned to
the edges incident to $v$.  Also, $\pi\,'$ is a
proper edge coloring of $T$ by setting
$\pi\,'(uu\,')=\pi\,'(vv\,')=\pi(uv)$ for $uu\,',vv\,'\in E(T)$
and $uv\in E\,'$; and $\pi\,'(e)=\pi(e)$ for $e\in E(H)\subset
E(T)$. Thereby, $k\leq
\chi\,'_{s}(T)$ since $k=\chi\,'_{s}(G)$. This corollary follows by Theorem \ref{them:main-theorem}.
\end{proof}

\begin{cor} \label{thm:GLT2-222}
Let $T$ be a spanning tree of a connected graph $G$, and $G[E\,']$ be an induced graph over the edge subset $E\,'=E(G)\setminus E(T)$. Then $\chi\,'_{s}(G)\leq \chi\,'_{s}(T)+\chi\,'(G[E\,'])$, where $\chi\,'(H)$ is the chromatic index of a graph $H$.
\end{cor}

As further work we present the following

\begin{conj} \label{conj:shan-yybb-chen}
Let $T$ be a tree with $2n_2(T)\leq (n_1(T)+k-1)^2$ for $k\geq 1$. Then $\chi \,'_s(T)\leq n_1(T)+k$.
\end{conj}

\vskip 0.4cm

\noindent \textbf{Acknowledgment.} The authors thank sincerely two referees' sharp opinions and helpful suggestions that improve greatly the article. B. Yao was supported by the National Natural Science
Foundation of China under Grant No. 61163054 and No. 61363060; X.-en Chen was supported by the National Natural Science Foundation of China under Grant No. 61163037.

\end{document}